\documentclass[12pt, a4paper, parskip=half, abstracton]{scrartcl}

\usepackage{array}
\usepackage{marginnote}
\usepackage{xcolor}
\usepackage{amscd,amssymb,amsfonts,amsmath,latexsym,amsthm}
\usepackage{hyperref}
\usepackage[all,cmtip]{xy}
\usepackage{mathrsfs}
\usepackage{graphicx}
\usepackage{bm} % for bold symbols in math mode \boldsymbol{}
\usepackage[nameinlink, capitalise]{cleveref}
\usepackage{etoolbox}
\def\hB{\hspace*{\fill}$\qed$}% \newline\noindent}

\usepackage[nottoc]{tocbibind} % Bibliography im toc

\usepackage{defs_pp1}
\usepackage{slashed}
\usepackage[utf8]{inputenc}
\usepackage{microtype}
\usepackage[english]{babel}
\usepackage{mathtools}
\usepackage{bm}
\usepackage{esvect}

\title{$G$-categories, bootstrap classes and UCT}
\author{
Ulrich Bunke\thanks{Fakult{\"a}t f{\"u}r Mathematik,
Universit{\"a}t Regensburg,
93040 Regensburg, 
GERMANY, \newline
ulrich.bunke@mathematik.uni-regensburg.de}  and 
 Benjamin Dünzinger \thanks{Fakult{\"a}t f{\"u}r Mathematik,
Universit{\"a}t Regensburg,
93040 Regensburg,
GERMANY, 
Benjamin.Duenzinger@mathematik.uni-regensburg.de %\newline
}
}

\numberwithin{equation}{section}
\newtheorem{theorem}{Theorem}					% Removed [section]!
\newtheorem{prop}[theorem]{Proposition}

\newtheorem{ddd}[theorem]{Definition}
\newtheorem{kor}[theorem]{Corollary}

\newtheorem{prob}[theorem]{Problem}

\theoremstyle{remark}
\theoremstyle{definition}

\newtheorem{rem}[theorem]{Remark}

\newcommand{\ee}{\mathrm{e}}

\newcommand{\fin}{\mathrm{fin}}

\newcommand{\EE}{\mathrm{E}}
\newcommand{\sepa}{\mathrm{sep}}

\newcommand{\nCalg}{C^{*}\mathbf{Alg}^{\mathrm{nu}}}

\newcommand{\Coind}{\mathrm{Coind}}

\newcommand{\Res}{\mathrm{Res}}

\newcommand{\Orb}{\mathbf{Orb}}

\newcommand{\CAlg}{{\mathbf{CAlg}}}

\newcommand{\ben}[1]{\textcolor{blue}{#1}}

\newcommand{\PSh}{{\mathbf{PSh}}}

 \newcommand{\Cat}{{\mathbf{Cat}}}

\newcommand{\Calg}{{\mathbf{C}^{\ast}\mathbf{Alg}}}

\renewcommand{\Pr}{\mathbf{Pr}}

\newcommand{\op}{\mathrm{op}}

\newcommand{\kk}{\mathrm{kk}}
\newcommand{\KK}{\mathrm{KK}}

\newcommand{\st}{\mathrm{st}}

\begin{document}  \maketitle \begin{abstract}  
 For finite groups $G$, we discuss the enrichment of equivariant $KK$- and $E$-theory in genuine $G$-spectra. 	\end{abstract}

  \tableofcontents
  
  \section{Introduction}In this note, we explain that equivariant $KK$- and $E$-theory for finite groups $G$ is naturally enriched in genuine $G$-spectra, or equivalently, spectral $G$-Mackey functors \cite{bw}. We further describe the corresponding bootstrap classes and derive universal coefficient and Künneth formulas. These results are formal consequences of the fact that equivariant $KK$- and $E$-theory refine to parametrized categories over the orbit category of $G$.
  
  \section{The bootstrap class and a Schwede-Shipley type theorem}
  
   Let $ \CAlg(\Pr^{L}_{\st})$ denote the $\infty$-category of presentably symmetric monoidal stable $\infty$-categories and left adjoint functors. In the following, we recall a classical fact (an $\infty$-categorical version of the Schwede-Shipley theorem \cite{Schwede_2003}), which holds for any object of $ \CAlg(\Pr^{L}_{\st})$ with a compact tensor unit.

 For concreteness, we formulate it as \cref{kopgwgrewfrw} for $\KK$ in $ \CAlg(\Pr^{L}_{\st})$ (equipped with either the minimal or maximal tensor product) representing the $KK$-theory of $C^{*}$-algebras. An analogous statement with an identical proof also holds for $\EE$ in $ \CAlg(\Pr^{L}_{\st})$ equipped with the maximal tensor product, or any other $\cC$ in $ \CAlg(\Pr^{L}_{\st})$ with a compact tensor unit. 
References for the construction of such versions of $\KK$ and $\EE$ are \cite{KKG}, \cite{Bunke:2023aa}, and \cite{Bunke:2026aa}. {We work here with the convention that $\KK = \Ind_{\aleph_1}(\KK_{\sepa})$ and $\EE = \Ind_{\aleph_1}(\EE_{\sepa})$. This follows the convention of \cite{Bunke:2026aa} and differs from the convention of \cite{KKG}. Furthermore, the canonical functors
\[ 
	\kk: \nCalg \to \KK \quad \text{and} \quad \ee: \nCalg \to \EE
\]
are given by applying the functor $\Ind_{\aleph_1}(-)$ to 
\[ 
\kk_\sepa: \nCalg_\sepa \to \KK_\sepa \quad \text{and} \quad \ee_\sepa: \nCalg_\sepa \to \EE_\sepa \ .
\]
This has the consequence that both $ \kk $ and $\ee$ map arbitary sums of $C^\ast$-algebras to coproducts. We will explain this choice in greater detail in \Cref{bobwaggqqihqffasf}.  
}

For any cocommutative coalgebra $A$ in $\KK$, the functor $$\map_{\KK}(A,-):\KK\to \Sp$$ has a lax symmetric monoidal refinement \cite{LurieHA}. An example of a cocommutative coalgebra (and also a commutative algebra) in $\KK$ is the tensor unit $\beins$. 
The endomorphism spectrum $$KU:=\map_{\KK}(\beins,\beins)$$ consequently admits the structure of a commutative ring spectrum,
and $\KK$ has a canonical enrichment in $\Mod_{KU}(\Sp)$. Formally, this can be expressed by stating that
$\KK$ in $\CAlg(\Pr^{L}_{\st})$ refines to an object of $ \CAlg( \Mod_{\Mod_{KU}(\Sp)}(\Pr^{L}_{\st}) )$. In particular,
 the lax symmetric monoidal functor $\map_{\KK}(\beins,-):\KK\to \Sp$ naturally refines to a lax symmetric monoidal functor
$$K:\KK\to \Mod_{KU}(\Sp)\ .$$

\begin{ddd} We define the bootstrap class $\cB$ as the localizing subcategory of $  \KK$ generated by $\beins$. \end{ddd} 
\begin{theorem}  \label{kopgwgrewfrw}\mbox{}
\begin{enumerate}
  \item  \label{kogwperfwf} For $A$ in $\cB$ and $B$ in $\KK$, the functor $K$ induces an equivalence
$$\map_{\KK}(A,B)\simeq \map_{\Mod_{KU}(\Sp)}(K(A),K(B))$$ in $\Mod_{KU}(\Sp)$.
\item \label{kgowperwrfwerfwef} The functor $K$ restricts to an equivalence $K_{|\cB}:\cB\stackrel{\simeq}{\to} \Mod_{KU}(\Sp)$.

\item \label{kgowperwrfwerfwee33f}  The functor $K$ is the right adjoint of a right Bousfield localization
$$b:\Mod_{KU}(\Sp)\rightleftarrows \KK:K\ .$$
\item \label{kogwperfwf1}  For $A$ in $\cB$ and $B$ in $\KK$, the canonical morphism $$K(A)\otimes_{KU} K( B) \to K(A\otimes B)$$ is an equivalence.
\end{enumerate}
\end{theorem}
Assertion \ref{kogwperfwf} represents the universal coefficient theorem (UCT), and Assertion \ref{kogwperfwf1} represents the Künneth formula.
 \begin{proof}[Sketch of proof]
 Since $K$ is corepresented, it preserves limits and is a right adjoint.
 Since $\beins$ is a compact object in $\KK$ (see \cite[Section 7]{Bunke:2026aa}), the functor $K$ preserves colimits. For any $B$ in $\KK$, the functor $K$ induces the identity
 $$K(B)\simeq \map_{\KK}(\beins,B)\stackrel{\simeq}{\to} \map_{\Mod_{KU}(\Sp)}(KU,K(B))\simeq K(B)\ .$$
This proves Assertion \ref{kogwperfwf} for $A=\beins$. The full assertion follows by extension along colimits.
This shows, in particular, that $K_{|\cB}$ is fully faithful. Since $K(\beins)\simeq KU$ and $K$ preserves colimits,
we conclude that $K_{|\cB}$ is essentially surjective, which implies Assertion \ref{kgowperwrfwerfwef}.
     Furthermore, it follows from \ref{kogwperfwf} that the left adjoint $b$ of $K$ maps $KU$ to $\beins$ and therefore maps $\Mod_{KU}(\Sp)$ fully faithfully into $\cB$. This establishes Assertion \ref{kgowperwrfwerfwee33f}.
     Finally, Assertion \ref{kogwperfwf1} is evident for $A=\beins$, and the general case again follows by extension along colimits. 
  \end{proof}
  
  {
  \begin{rem}\label{bobwaggqqihqffasf}
 In advantage of the conventions for $KK$-theory used in the present note, compared with the ones from  \cite{KKG}, is that the intersection of the  bootstrap class with the $\aleph_{1}$-compact objects coincides with the classical bootstrap class. To explain this, let us denote the $KK$-theory as considered in \cite{KKG} by   $\overline{\KK}:= \Ind(\KK_{\sepa})$.  {W}e consider the functor
  	\[ 
  		\overline{\kk}: \nCalg_\sepa \to \KK_\sepa \xrightarrow{y} \Ind(\KK_\sepa) = \overline{\KK} \ .
  	\]
  	Then, $\overline{\kk} {(C_0(\nat))}$ does not lie in the bootstrap class in $ \overline{\KK} $. {Assuming the contrary, the object} $\kk_\sepa {(C_0(\nat))}$  {would lie} in the smallest idempotent complete stable subcategory  {of $\KK_{\sepa}$} containing the tensor unit, as  $\overline{\kk}{(C_0(\nat)})$ is compact. {Then  
	 the homotopy groups of $K(C_0(\nat)) $ were}  finitely generated, which is clearly a contradiction.
  	
	This problem does not appear when working with $\KK$.  {Indeed, the object $\kk(C_0(\nat)) \simeq \bigoplus_\nat \beins$
	  clearly belongs to the bootstrap class of $\KK$.} \hB
  \end{rem}
	}

\section{$G$-categories and Mackey functors}
 
  In the following, we explain how \cref{kopgwgrewfrw} can be extended to 
 equivariant $KK$-theory for finite groups. The natural framework is 
 the theory of parametrized categories over $G\Orb$. The parametrized version of a presentably symmetric monoidal stable $\infty$-category is a 
 $G$-presentably symmetric monoidal $G$-stable $\infty$-category.
  Following \cite{Cnossen:2024aa}:
   \begin{ddd}\label{kopgwgwegw} A $G$-presentably symmetric monoidal $G$-stable $\infty$-category is a functor
 $\cC:\PSh(G\Orb)^{\op}\to \CAlg(\Pr^{L}_{\st})$ satisfying the following properties:  
 \begin{enumerate}
 \item \label{okgpwgerwgwerg} $\cC$ is limit preserving.
 \item \label{koprgwegerfw} $G$-presentable: For any map $f:T\to S$ between finite $G$-sets, the restriction functors $f^{*}:\cC({S})\to \cC({T})$ have left adjoints $f_{!}$ such that the corresponding Beck-Chevalley maps are equivalences.%conditions.
 \item \label{koprgwegerfw1} $G$-stable: This condition is equivalent to $G$-semiadditivity, meaning that the restriction functors $f^{*}$ also possess right adjoints such that the canonical norm maps $f_{!}\to f_{*}$ are equivalences.
 \item \label{okpwegergweg} $G$-presentably symmetric monoidal: The induction functors $f_{!}$ satisfy the projection formula.
 \end{enumerate}\end{ddd}
 
 Due to Assumption \ref{okgpwgerwgwerg}, the functor $\cC$ is uniquely determined by its restriction
$$\cC_{|G\Orb^{\op}}:G\Orb^{\op}\to  \CAlg(\Pr^{L}_{\st})\ .$$

In Assumption \ref{koprgwegerfw}, we treat
a finite $G$-set $S$ as a presheaf $G\Orb^{\op}\ni X\mapsto \Hom_{G\Set}(X,S)$ on the orbit category.
For any cartesian square 
$$\xymatrix{S\ar[r]^{g}\ar[d]^{l} &T \ar[d]^{f} \\U \ar[r]^{h} &V } $$ in $G\Set^{\fin}$,
there is a canonical Beck-Chevalley map $l_{!}g^{*}\to h^{*}f_{!}:\cC(T)\to \cC(U)$ which is required to be an equivalence. {Here, $G\Set^{\fin}$ is the category of finite sets with $G$-action.}

For the map $f_{!}:S\to *$ of finite $G$-sets, the norm map in Assertion \ref{koprgwegerfw1}
is essentially a transformation $\coprod_{S}=f_{!}\to f_{*}=\prod_{S}$. Consequently, we require that 
coproducts indexed by finite $G$-sets are equivalent to the corresponding products.

 The projection formula in Assumption \ref{okpwegergweg} asserts that, for any map $f:S\to T$ between finite $G$-sets, the natural map $f_{!}(A\otimes f^{*}B)\to f_{!}(A)\otimes B$ is an equivalence for any $A$ in $\cC(S)$ and $B$ in $\cC(T)$.
 
 {In the following, we explain our convention of equivariant $KK$- and $EE$-theory. We define the separable versions $\KK^G_\sepa$ and $\EE^G_\sepa$ as in \cite{KKG} and \cite{Bunke:2026aa}, respectively. The categories $\KK^G$ and $\EE^G$ are defined as the $\Ind_{\aleph_1}$-completion of their separable counterparts.}

 Following \cite{Cnossen:2023aa}:
\begin{prop} There exist $G$-presentably symmetric monoidal $G$-stable $\infty$-categories
$ \KK_{G}$ and $\EE_{G}$  
such that
$\KK_{G}(G/H)\simeq \KK^{H}$ for all subgroups $H$ of $G$ (and similarly for $E$-theory), where the functoriality is induced by the restriction maps.
\end{prop}
\begin{proof}[Sketch of proof]
For concreteness, we focus on $KK$-theory. The argument for $E$-theory is analogous. One
applies the methods developed in \cite{KKG} to the symmetric monoidal $G$-category $$\nCalg_{G}:G\Orb^{\op}\to \CAlg(\Cat)$$
of $C^{*}$-algebras with $\nCalg_{G}(G/H)\simeq H\nCalg$. 
Initially, one considers the $G$-subcategory of separable algebras, which is then Dwyer-Kan localized at the equivariant $KK$-equivalences, followed by {applying the ind completion functor $\Ind_{\aleph_1}$}. {This process yields a well-defined functor $\KK_G:\Orb^\op \to \Calg(\Cat)$, because}  
the restriction functors $K\nCalg\to H\nCalg$ for all subgroups $H\subseteq K\subseteq G$ preserve
equivariant $\KK$-equivalences.
%, the global Dwyer-Kan localization induces
%the Dwyer-Kan localization at the $\KK$-equivalences pointwise at every evaluation. This demonstrates that the resulting functor $\KK_{G}$
%exhibits the correct evaluations. 
To verify $G$-stability, one checks that the $KK$-level restriction functors $\Res^K_H$ admit left and right adjoints $\Ind^K_H$ and $\Coind^K_H$, which naturally happen to be equivalent. The Beck-Chevalley condition and projection formula can be verified for $\Coind$ directly on the level of $C^*$-algebras. 
\end{proof}

{
\begin{rem}
	Here, we deviated slightly from \cite{Cnossen:2023aa}. The difference is that \textit{loc. cit.} works with $\Ind(\KK_\sepa)$ instead of with $ \Ind_{\aleph_1}(\KK_\sepa^G) = \KK^G $. This change  {does not} affect the proof. \hB
\end{rem}
}

%\ben{
%\begin{rem}
%	Instead of applying $\Ind$ in the above proof one can also apply the $\Ind_{\aleph_1}$-completion functor, let us call the resulting $G$-category in the case of $KK$-theory $\widetilde{\KK}_G$ and it's evaluation at the $G$-set $G/H$ by $ \widetilde{\KK}^G $. The main difference between $\widetilde{\KK}^G$ and $ \KK^G $ is that the canonical functor $ \widetilde{\kk}^G: G\nCalg \to \widetilde{\KK}^G $ preserves all sums (see \cite[Remark 3.14.4]{budu}) and $\aleph_1$-filtered colimits, whereas the functor $\kk^G: G\nCalg \to \KK^G$ only preserves $\aleph_1$-filtered colimits. Consequently, typically it seems preferably to work with $\widetilde{\KK}_G$ instead of $\KK_G$.  Nonetheless, we have chosen to work with $\KK_G$ instead of $\widetilde{\KK}_G$ to ensure that this note is compatible with \cite{Cnossen:2023aa} and \cite{KKG}. But  the same arguments presented here for $\KK_G$  can also be applied verbatim to $\widetilde{\KK}_G$ except for the proof of the compactness of tensor unit in $\widetilde{\KK}^G$. In the case of $\KK^G$ the compactness of the tensor unit is a formality, whereas in the case of $\widetilde{\KK}^G$ it follows by the Green--Julg theorem \cite[Theorem 4.25]{KKG}, combined with the compactness of the tensor unit in $\widetilde{\KK}$, and the fact that $-\rtimes G $ preserves all sums. For a detailed analysis of the non-equivariant category $\tilde{\KK}$ we refer the reader to \cite{Bunke:2026aa}. The same discussion also applies to equivariant $E$-theory.
%\end{rem}
%}

Returning to the general setting, we define $\cC^{H}:=\cC(G/H)$ and refer to
 $\cC^{G}:=\cC(G/G)$ as the underlying presentably symmetric monoidal stable $\infty$-category. 
 Let $\beins_{G}$ in $\cC^{G}$ denote the tensor unit.  {In the following, we use terminology for structures associated to $\cC$ that are motivated by the relevant examples $\EE_G$ and $\KK_G$. This in particular applies to the representation ring and the $K$-theory functor introduced below.} 
  
The $G$-category of genuine $G$-spectra $ \Sp_{G}$ is the initial 
 $G$-presentably symmetric monoidal $G$-stable $\infty$-category. Any other $G$-presentably symmetric monoidal $G$-stable $\infty$-category $\cC$ is enriched over
 $\Sp_{G}$. In particular,  {there exists} a bifunctor
 $$ \underline{\map}_{\cC}:\cC^{\op}\times \cC\to   \Sp_{G}\ .$$
 %Inserting the tensor unit in the first slot we get a lax symmetric monoidal right-adjoint the functor
 %$$\underline{K}:\cC\to \Sp_{G}\ .$$
 The evaluation
   $\Sp^{G}:= \Sp_{G}(*)$ represents the classical category of genuine $G$-spectra (or spectral $G$-Mackey functors),
   and we define the representation ring as
   $$R(G):= \underline{\map}_{\cC}(\beins_{G},\beins_{G})(*)\in \CAlg(\Sp^{G})\ .$$
This yields a 
 $G$-presentably symmetric monoidal $G$-stable $\infty$-category
 $  \Mod_{ R(G)}(\Sp_{G})$ alongside a $G$-limit preserving refinement
$$\underline{K}:\cC\to \Mod_{ R(G)}(\Sp_{G})$$
of the functor 
$\underline{\map}(\beins_{G},-)$.
Let $$K^{G}:\cC^{G}\to \Mod_{ R(G)}(\Sp_{G})^{G}:= \Mod_{ R(G)}(\Sp_{G})(*)$$ be the evaluation of $\underline{K}$ at $*$.
The functor $\underline{K}$ satisfies the following properties:
\begin{enumerate}
\item We have $\underline{K}(\beins_{G})\simeq R(G)$, and for any $A$ in $\cC^{G}$, the map induced by 
$$K^{G}(A)\simeq  \underline{\map}_{\cC}(\beins_{G},A)(*)\to \underline{\map}_{ \Mod_{R(G) }(\Sp_{G})}(R(G),\underline{K}(A))(*)\simeq K^{G}(A)$$ is the identity.
%\item More generally we have $$  \underline{\map}(\beins_{G},A)(*)\stackrel{\simeq}{\to} \map_{ \Mod_{\Sp^{G}}(R(G))}(R(G), K (A))\ .$$
\item If $\beins_{G}$ is compact in $\cC^{G}$, then $\underline{K}$ preserves $G$-colimits.
\end{enumerate}

\begin{ddd} We define the bootstrap class $\cB$ to be the $G$-subcategory of $ \cC$ generated by $\beins_{G}$ under $G$-colimits and shifts. 
\end{ddd}Then $\cB^{G}:=\cB(*)$ is generated by $f_{!}f^{*}\beins_{G}$ for all maps $f:S\to *$ originating from finite $G$-sets  {as a localizing subcategory}.   
Using the aforementioned properties, arguments analogous to the proof of \cref{kopgwgrewfrw}
yield:
\begin{theorem}\mbox{} Assume that $\beins_{G}$ is compact.
\begin{enumerate}
  \item  \label{kogwperfwf} For $A$ in $\cB^{G}$ and $B$ in $\cC^{G}$, we have an equivalence
$$ \underline{\map}_{\cC}(A,B)\simeq \underline{\map}_{\Mod_{R(G)}(\Sp_{G})}(\underline{K}(A),\underline{K}(B)) \ .$$
\item \label{kgowperwrfwerfwef} The functor $\underline{K}$ restricts to an equivalence $\underline{K}_{|\cB}:\cB\stackrel{\simeq}{\to} \Mod_{R(G)}(\Sp_{G})$.
 \end{enumerate}
\end{theorem}

By evaluating these equivalences at $*$, we arrive at the following corollary via the same logic applied in \cref{kopgwgrewfrw}: \begin{kor}\label{jigjowgsfgfdgsfdg}\mbox{}
\begin{enumerate}
  \item  \label{kogwperfwf} For $A$ in $\cB^{G}$ and $B$ in $\cC^{G}$, we have an equivalence
$$ \underline{\map}_{\cC}(A,B)(*)\simeq \map_{\Mod_{R(G)}(\Sp^{G})^{G}}(K^{G}(A),K^{G}(B))$$ in $\Mod_{R(G)}(\Sp^{G})^{G}$.
\item \label{kgowperwrfwerfwef} The functor $  K^{G}$ induces an equivalence $K^{G}_{|\cB}:\cB^{G}\stackrel{\simeq}{\to} \Mod_{R(G)}(\Sp^{G})^{G}$.
\item The functor $  K^{G}$ is the right adjoint of a  {right Bousfield localization} 
  $$b: \Mod_{R(G)}(\Sp^{G})^{G}\rightleftarrows \cC^{G}:K^{G}\ .$$
\item For any $A$ in $\cB^{G}$ and $B$ in $\cC^{G}$, the canonical map
$$K^{G}(A)\otimes_{R(G)} K^{G}(B)\to K^{G}(A\otimes B)$$ is an equivalence.
 \end{enumerate}
\end{kor}

We now focus specifically on the case where $\cC:=\KK_{G}$.  {We know that $\beins_G$ is compact, because ${\beins}$ in $\KK$ is compact, we have the equivalence $\map_{\KK^G}({\beins}_G, -) \simeq \map_{\KK}({\beins}, - \rtimes G)$ by the  the Green-Julg theorem \cite[Theorem 4.25]{KKG}, and since the functor ${-}\rtimes G{:\KK^{G}\to \KK}$ preserves sums.}
%We know that $\beins_{G}$ is compact \cite{MeyerNest2006}. 
Let $KU^{G}:=R(G)$ in $\Sp^{G}$ denote the genuine $G$-equivariant $K$-theory spectrum.
  \cref{jigjowgsfgfdgsfdg} states that
the $K$-theory functor $$K^{G}:\KK^{G}\to \Mod_{KU^{G}}(\Sp^{G})(KU^{G})^{G}$$ identifies the bootstrap class $\cB^{G}$ with the $KU^{G}$-modules in genuine $G$-spectra. Thus, the bootstrap class is generated by the objects $\Ind_{H}^{G}(\Res_{H}^{G}(\beins_{G}))$ for all subgroups $H$ of $G$.
The universal coefficient theorem states that
$$\KK^{G}(A,B)\simeq \map_{\Mod_{R(G)}(\Sp^{G})}(K^{G}(A),K^{G}(B))$$ for $A$ in the bootstrap class $\cB^{G}$ and $B$ in $\KK^{G}$, where the right-hand side is the genuine mapping spectrum in
$R(G)$-modules between the genuine $K$-theory spectra of $A$ and $B$.

 We now observe that $\KK^{G}$ is in fact a module category over $\KK$. In \cref{kopgwgwegw},
 we can substitute $ \CAlg(\Pr^{L}_{\st})$ with $\CAlg(\Mod_{KK}(\Pr^{L}_{\st}))$ and regard
 $ \KK_{G}$ as a $G$-presentably symmetric monoidal $G$-stable $\KK$-module category.
 This yields a refined representation ring $\smash{\tilde{KU}^{G}}$ in $\Sp^{G}\otimes \KK$ (or, expressed differently, a commutative algebra in $\KK$-valued Mackey functors). The restriction of the corresponding $K$-theory functor
 $$\tilde K^{G}:\KK^{G}\to \Mod_{\smash{\tilde{KU}^{G}}}(\Sp^{G}\otimes \KK)$$ to the generalized bootstrap class $\tilde \cB^{G}$
 is an equivalence, where $\tilde \cB^{G}$ is generated by $\Ind_{H}^{G}(\ben{A})$, where $A$ is an algebra with trivial $H$-action,  for all subgroups $H$ of $G$.
 
 \begin{prob}
 Is the functor $ \tilde K^{G}: \KK^{G}\to   \Mod_{\smash{\tilde{KU}^{G}}}(\Sp^{G}\otimes \KK)  $ an equivalence?
 \end{prob}
  
 An equivalent formulation of this problem is:
 \begin{prob}
 Do there exist objects {$A \not \simeq 0$} in $\KK^{G}$ such that $\Res^{G}_{H}(A)\rtimes_{r}H\simeq 0$ in $\KK$ for all subgroups $H$ of $G$?
 \end{prob}

%
%
%Then $\underline{K}_{|\cB}:\cB\to  \Mod_{\Sp_{G}}(R(G))$ is an equivalence.
%In particular, it induces an equivalence $$\cB^{G}\stackrel{\simeq}{\to}   \Mod_{\Sp_{G}}(R(G))^{G}\simeq \Mod_{\Sp^{G}}(R(G))$$
%and we have a right Bousfield localization
%$$b:\cB^{G}\rightleftarrows: \cC^{G}:K\ .$$
%
%Replacing $\Pr^{L}_{\st}$ by $\Mod_{\Pr^{L}_{\st}}(\KK)$ one can obtain an 
%equivalence  $$\tilde \cB^{G}\stackrel{\simeq}{\to}   \Mod_{\Sp_{G}\otimes \KK}(\tilde R(G))^{G}\simeq \Mod_{\Sp^{G}\otimes \KK }(\tilde R(G))$$
%(and $\Sp^{G}\otimes \KK$ are $\KK$-valued Mackey functors)
%$\tilde \cB^{G}$ is generated by $B\otimes A$ with $B$ in $\cB^{G}$ and $A$ in $KK$
%and an 
%adjunction
%$$\tilde b:\tilde \cB^{G}\rightleftarrows \cC^{G}:\tilde K$$
% 
%
%
%\cite{Cnossen:2024aa},\cite{Cnossen:2023aa}
%

    \bibliographystyle{alpha}
\bibliography{forschung2021}

\end{document}